\documentclass[11pt]{article}
\usepackage{sustyle}
\usepackage{tikz}
\usetikzlibrary{decorations.text,shapes,snakes,arrows.meta}

\usepackage{tcolorbox}
\usepackage[hidelinks]{hyperref}
\usepackage{algorithm}
\usepackage[noend]{algpseudocode}

\usepackage[T1]{fontenc}

\usepackage{setspace}
\usepackage{pgfplots}
\usepackage{bchart}
\def\d{\mathrm{d}}

\usepackage{verbatim}

\DeclareMathOperator{\Tr}{Tr}
\DeclareMathOperator{\msgn}{msgn}

\newcommand{\citep}{\cite}
\newcommand{\citet}{\cite}
\newcommand{\dW}{\delta W}

\title{Isotropic Curvature Model for Understanding Deep Learning Optimization: Is Gradient Orthogonalization Optimal?}

\begin{document}

\author{Weijie Su\\[0.5em]University of Pennsylvania}
\date{October 31, 2025}

\maketitle

\begin{abstract}

In this paper, we introduce a model for analyzing deep learning optimization over a single iteration by leveraging the matrix structure of the weights. We derive the model by assuming isotropy of curvature, including the second-order Hessian and higher-order terms, of the loss function across all perturbation directions; hence, we call it the isotropic curvature model. This model is a convex optimization program amenable to analysis, which allows us to understand how an update on the weights in the form of a matrix relates to the change in the total loss function. As an application, we use the isotropic curvature model to analyze the recently introduced Muon optimizer and other matrix-gradient methods for training language models. First, we show that under a general growth condition on the curvature, the optimal update matrix is obtained by making the spectrum of the original gradient matrix more homogeneous---that is, making its singular values closer in ratio---which in particular improves the conditioning of the update matrix. Next, we show that the orthogonalized gradient becomes optimal for the isotropic curvature model when the curvature exhibits a phase transition in growth. Taken together, these results suggest that the gradient orthogonalization employed in Muon and other related methods is directionally correct but may not be strictly optimal. Finally, we discuss future research on how to leverage the isotropic curvature model for designing new optimization methods for training deep learning and language models.

\end{abstract}

\section{Introduction}
\label{sec:introduction}

If asked before 2025 which optimization method is used for training deep learning and language models, most practitioners would have answered Adam \cite{kingma2015}. Indeed, since its proposal in 2014, Adam has become almost synonymous with deep learning optimization. Although many potential competitors were introduced, they either failed to scale their performance to larger models, worked well only with sophisticatedly tuned hyperparameters, or the marginal improvement was not significant enough to motivate practitioners to switch from Adam \cite{sun2020optimization}. The consensus in both industry and academia was that Adam's dominance was secure for years to come. For instance, the work on Adam received the Test of Time Award in early 2025 from the International Conference on Learning Representations (ICLR), one of the premier machine learning conferences.

However, Adam's dominance was challenged by a new optimizer called Muon \cite{jordan2024muon} in December 2024, which was shown by its developers to outperform Adam and other existing optimizers on relatively small-scale language models. As with most optimizers when first proposed, many questioned whether the superior performance of Muon could scale to larger, commercial-grade language models. Indeed, many, including myself, were skeptical, partly because Muon's structure is less amenable to parallelization than Adam's. However, this skepticism quickly subsided as a host of Muon variants were developed to incorporate various practical considerations \cite{pethick2025training,an2025asgo,ahn2025dion,riabinin2025gluon,lau2025polargrad,he2025low,zhang2025adagrad,li2025normuon,he2025demuon,huang2025limuon,khaled2025muonbp,gruntkowska2025drop} and its superiority was quickly demonstrated on larger, industry-scale language models \cite{liu2025muon,essentialai2025practical,kimi2025k2}. In particular, in February 2025, Muon was shown to require only about 52\% of the FLOPs to achieve the same empirical performance as AdamW \cite{loshchilov2019decoupled}, a popular variant of Adam, on a 16-billion-parameter language model using 5.7 trillion tokens \cite{liu2025muon}. In July, \cite{kimi2025k2} used Muon to train a frontier language model with 32 billion activated parameters and 1 trillion total parameters. Mounting evidence now suggests that Muon is becoming, and perhaps already is, the new go-to optimizer for training language models in industry, which is remarkable especially given that it all happened in less than one year.

Interestingly, the operational mechanism of Muon is both intuitive and, for many trained in mathematical optimization, surprising. The intuitive part is that Muon recognizes that weights in deep learning architectures are structured as matrices, such as in a feedforward layer or the weight matrices for query, key, and value in an attention mechanism. This is in contrast to Adam, which neglects this matrix structure by vectorizing all weights. To introduce Muon, let $f(W)$ denote the total loss function to be minimized, where $W$ denotes parameters in the form of a matrix (we omit other parameters for simplicity). The gradient on a mini-batch of training samples, denoted $G$, is also a matrix of the same size. The Muon optimizer updates the weights not along the direction of the original gradient $G$, but rather along the direction of the orthogonalized gradient. Formally, let $G = U \Sigma V^\top$ be the (compact) singular value decomposition (SVD);\footnote{In practice, the orthogonalization is applied to an exponential moving average of the gradients.} then
\[
W \leftarrow W - \gamma U V^\top,
\]
where $\gamma$ is the learning rate. Equivalently, the update direction $U V^\top$, often denoted $\msgn(G)$, is the unitary part from the polar decomposition of the gradient $G$ \cite{amsel2025polar,lau2025polargrad,crawshaw2025exploration}; a variant of Muon was therefore named PolarGrad to better reflect its connection to classical numerical linear algebra \cite{lau2025polargrad}. While Muon is not the first optimizer to incorporate the matrix structure of gradients \cite{carlson2015preconditioned,carlson2015stochasticRBM,gupta2018shampoo,tuddenham2022orthogonalising,vyas2024soap}, a notable design feature of Muon is its use of polynomial methods to approximate the unitary matrix $U V^\top$ efficiently in a GPU environment.

The immediate surprise, however, is why all singular values should be made constant. A larger singular value suggests that the corresponding singular space is more ``promising'' for reducing the loss function. Discarding the singular value matrix $\Sigma$ seems to ignore the varying levels of promise across different singular spaces. Indeed, taking a step back, it is surprising that singular values should be made closer at all, let alone perfectly uniform. A secondary but non-trivial question is why the singular spaces of the original gradient matrix should be preserved in the update. These questions are pressing, as without a theoretical foundation, the practice of deep learning optimization risks becoming an exercise in trial and error. To address these puzzles, a surge of work has emerged to understand Muon \cite{li2025muon,kovalev2025understanding,shen2025convergence,sfyraki2025lions,sato2025analysis,chang2025convergence,zhang2025provable}, much of which adopts a spectral norm perspective \cite{bernstein2024old,bernstein2024modular,fan2025implicit,lau2025polargrad,chen2025muon}. However, it is still unclear how gradient orthogonalization contributes to enhanced optimization performance, and whether it is really optimal from a theoretical viewpoint.

In this paper, we propose an analytically tractable optimization program to model Muon and, more generally, matrix-gradient methods. Our model focuses on finding an update that achieves the steepest descent in an ``average'' single iteration. Our rationale is that, to model the change in loss, we can assume that beyond the first-order information (the gradient), all higher-order information, including the Hessian and other curvature information, is isotropic across all possible directions. This isotropy assumption is motivated by the immense scale of deep learning models, especially language models, where the architecture and training process do not inherently favor any particular direction. Furthermore, by averaging over a full batch with a very large number of training samples, curvature might exhibit a convergent law. We also impose certain conditions on how the curvature depends on the perturbation magnitude that are consistent with experimental observations. We call this framework the isotropic curvature model.

By analyzing the isotropic curvature model, we obtain several insights into the effective design of matrix-gradient methods. First, the isotropy of our model's design justifies the preservation of the singular spaces of the original gradient matrix in the update. More importantly, our model offers guidance on how the spectrum of the gradient matrix should be modified. By assuming a certain growth condition on the curvature, our model shows that the optimal update matrix is obtained by pushing the singular values of the original gradient matrix closer to each other while preserving their original ordering, a property we refer to as spectrum homogenization. Next, we prove that when the curvature in the isotropic curvature model transitions abruptly from small to large, optimality is attained by setting all singular values to be equal; that is, orthogonalization is optimal in this asymptotic limit.

Several implications follow from our analysis of the isotropic curvature model. Roughly speaking, the model suggests that the design rationale of Muon and many other matrix-gradient methods is conceptually sound, as they respect the matrix structure of the gradients. However, our model implies that Muon is not strictly optimal. The optimal spectrum transformation is never perfectly uniform, because in practice the curvature does not exhibit the asymptotic behavior required for the optimality of orthogonalization. Determining the optimal spectrum depends on specifying the curvature function, and a possible future direction for designing matrix-gradient optimizers is to first approximate the curvature information and then find the optimal update matrix by solving a simple but large-scale convex optimization program.

\section{Isotropic Curvature Model}
\label{sec:xxx-model}

\subsection{Derivation}

Consider the objective function that we wish to minimize:
\[
f(W) = \frac1N \sum_{i=1}^N L(W; x_i, y_i),
\]
where $W$ is a matrix parameter of size $m \times n$, $L$ is a (cross-entropy) loss function, and $(x_i, y_i)$ is the $i$-th training example and its label. This formulation includes the scenario where $x_i$ is the context and $y_i$ is the token for training large language models. The full objective function involves many other parameters, but here we isolate $W$ to study how its update affects the optimization process. Let $u_i$ be the intermediate data from $x_i$ that serves as input to the matrix $W$. Recognizing that the loss function depends on $x_i$ through the product $W u_i$, we can write the loss as $L(W u_i)$.\footnote{Here we ignore residual connections, so the function depends on $u_i$ only through $W u_i$. We also omit the dependence on the label $y_i$ for simplicity.}

Next, we consider how the objective $f$ changes when $W$ is updated to $W + \dW$ for some small matrix perturbation $\dW$. A good optimizer should find a $\dW$ that reduces $f(W + \dW)$ as much as possible. We Taylor expand the loss function $L$ for each sample, rather than the objective $f$, as this allows us to leverage the matrix structure of the weights $W$. Assuming sufficient smoothness, we expand the loss function up to the quadratic term with a remainder:
\[
\begin{aligned}
L(Wu_i + \dW u_i) =&  L(Wu_i) + \langle \nabla L(Wu_i), \dW u_i \rangle \\
&+ (\dW u_i)^\top \left[ \int_0^1 (1 - t) \nabla^2 L(Wu_i + t \dW u_i) \d t \right] \dW u_i,
\end{aligned}
\]
where $\nabla L$ is the gradient with respect to the entire input argument of $L(\cdot)$. Note that $\int_0^1 (1 - t)\nabla^2 L(x + t \delta x) \d t$ is a weighted average of the Hessians of $L$ along the line segment connecting $x$ and $x + \delta x$. Plugging this into the objective function gives
\[
\begin{aligned}
f(W + \dW)  =& f(W) + \frac1N \sum_{i=1}^N \langle \nabla L(W u_i), \dW u_i \rangle \\
&+ \frac1{N} \sum_{i=1}^N (\dW u_i)^\top \left[ \int_0^1  (1 - t)\nabla^2 L(W u_i + t \dW u_i)\d t \right] \dW u_i.
\end{aligned}
\]
The first-order term simplifies to
\[
\sum_{i=1}^N \langle \nabla L(W u_i), \dW u_i \rangle = \sum_{i=1}^N \langle \dW, \nabla L(W u_i) u_i^\top \rangle =N \langle \dW, \nabla f(W) \rangle,
\]
where $\langle \cdot, \cdot \rangle$ denotes the Frobenius inner product. Hence,
\[
\begin{aligned}
f(W + \dW)  =& f(W) + \Tr(\dW \nabla f(W)^\top) \\
&+ \frac1{N} \sum_{i=1}^N (\dW u_i)^\top \left[ \int_0^1 (1 - t) \nabla^2 L(W u_i + t \dW u_i)  \d t \right] \dW u_i\\
\equiv& f(W) + \Tr(\dW \nabla f(W)^\top) + \frac1{N} \sum_{i=1}^N h(\dW u_i, Wu_i).
\end{aligned}
\]

We approximate $f(W + \dW)$ by replacing the term $\frac1{N} \sum_{i=1}^N h(\dW u_i, Wu_i)$ with a simpler expression that allows us to analyze the optimal $\dW$ within a surrogate model. First, we assume that the singular spaces of the Hessian integral $\int_0^1 (1 - t) \nabla^2 L(W u_i + t \dW u_i)  \d t$, which can be viewed as random due to factors like initialization and mini-batch sampling, are independent of the vector $\dW u_i$. This allows us to approximate
\begin{equation}\label{eq:h_w_w}
\begin{aligned}
h(\dW u_i, Wu_i) &\equiv (\dW u_i)^\top \left[ \int_0^1 (1 - t) \nabla^2 L(W u_i + t \dW u_i)  \d t \right] \dW u_i\\
& \approx \text{average eigenvalue of}\left[\int_0^1 (1 - t)\nabla^2 L(W u_i + t \dW u_i) \d t \right] \times \|\dW u_i\|^2,
\end{aligned}
\end{equation}
where $\|\cdot\|$ throughout the paper denotes the $\ell_2$ norm. It is worth noting that this approximation might not be precise for an individual $i$, but it becomes more faithful when averaged over all $i$.

To further simplify $\frac1{N} \sum_{i=1}^N h(\dW u_i, Wu_i)$, we need to make assumptions on the average spectrum of the Hessian integral. To this end, we impose the \textit{isotropy} assumption that the spectral distribution depends on the perturbation $\dW u_i$ only through its Euclidean norm $\|\dW u_i\|$. This assumption is motivated by the fact that deep learning models, especially large language models, are vast systems where no direction is favored by design. Moreover, weights are often initialized from isotropic Gaussian distributions, and no specific direction is explicitly favored during training. To further simplify, we drop the dependence of $h(\dW u_i, Wu_i)$ on $Wu_i$, which is a limitation of our model but may be plausible since we are conditioning on the current weights $W$. Taken together, under our isotropic assumption, both the average spectrum and $\|\dW u_i\|^2$ in \eqref{eq:h_w_w} are fully determined by $\|\dW u_i\|$. Hence, we introduce a \textit{curvature function} $H$ such that $H(\|\dW u_i\|) \approx h(\dW u_i, Wu_i)$.

This leads to the following optimization program for modeling the one-step update:
\[
\min_{\dW} f(W + \dW)  \approx f(W) + \Tr(\dW \nabla f(W)^\top) + \frac1N \sum_{i=1}^N H(\|\dW u_i\|).
\]
However, the presence of the input data vectors $u_i$ still complicates the analysis. To overcome this, we make another isotropy assumption by assuming all $u_i$ are independent and identically distributed samples from a sphere centered at the origin. Again, the rationale for this isotropy is that no input directions are favored and the data are diverse enough to plausibly span all directions. This allows us to approximate the sum with an expectation: 
\[
\frac1N \sum_{i=1}^N H(\|\dW u_i\|) \approx \E_{\zeta \sim \text{sphere}} H(\|\dW \zeta\|).
\]

Finally, for notational convenience, we let $Q = - \dW$ denote the negative update direction (so the updated weights are $W - Q$). We use $G = \nabla f(W)$ to denote the gradient, which in practice is a stochastic gradient from a mini-batch. The update matrix $Q$ specifies not only the direction but also incorporates the learning rate. This yields the following optimization program:
\begin{equation}\label{eq:our_model}
\min_{Q} -\Tr(Q G^\top) + \E_{\zeta \sim \text{sphere}} H(\|Q \zeta\|).
\end{equation}
We call this the \textit{isotropic curvature model}. Since the radius of the sphere can be absorbed into $H$, we assume for simplicity that $\zeta$ is sampled from the unit sphere in $\R^n$. We remark that this model is not intended to capture a specific iteration in deep learning optimization, which is subject to randomness and arbitrariness. Instead, it seeks to model, in an average or distributional sense, how a single iteration proceeds.

\subsection{Assumptions on Curvature}
\label{sec:grow}

The simplicity of the isotropic curvature model \eqref{eq:our_model} is that it represents complex high-order information as a single univariate, increasing function $H$. To analyze this model, however, we need to make a further assumption on the curvature function $H$. Since $H$ contains all high-order terms, we can assume it grows rapidly. Indeed, to ensure that \eqref{eq:our_model} is well-defined, we need to assume, for example, that $H(r)/r > C$ for some sufficiently large constant $C$ (depending on $G$) when $r$ is large enough. This is almost a minimal requirement for the optimization program to have a well-defined optimal solution. Otherwise, the objective function $-\Tr(Q G^\top) + \E_{\zeta \sim \text{sphere}} H(\|Q \zeta\|)$ could diverge to $-\infty$. For instance, if $G = C' I$, one could take $Q = cI$ for a large constant $C'$ and let $c \to \infty$.\footnote{Similar examples can be constructed for non-square problems.} In practice, $H(r)$ does not grow to infinity as $r \to \infty$; for a very large $r$, the loss would plateau at the level of a random predictor, which is finite. To ensure the uniqueness of the solution in the generic case, we adopt another technical assumption that $H$ is convex. This ensures that \eqref{eq:our_model} is a convex program in $Q$, because $\|Q \zeta\|$ is a convex function of $Q$ and $H$ is an increasing convex function, making their composition $H(\|Q \zeta\|)$ convex.

Next, we examine the growth of $H$ more closely. When $r > 0$ is very small, the high-order information is dominated by the second-order term (the Hessian). Hence, it is reasonable to assume $H(r) \approx c r^2$ for sufficiently small $r$. When $r$ becomes larger, it is no longer reasonable to assume a constant average of all eigenvalues for the Hessian integral. To account for this, we may assume a power law $H(r) \approx c' r^{2 + \alpha}$ and leave the determination of the constant $\alpha$ to empirical experiments. 

From experiments on GPT-2 \cite{radford2019language}, as shown in Figure \ref{fig:growth}, we plot approximations of $H(r)/r^2$. When $r$ is small, $H(r)/r^2$ is nearly constant, indicating that a quadratic approximation holds. When $r$ exceeds approximately $10^{0.5}$, $H(r)/r^2$ grows roughly as a power law $r^{\alpha}$, with a typical positive $\alpha$ around $0.2$. This shows that the high-order terms grow at a super-quadratic rate in practical problems. This is not unexpected, as for larger $r$, third-order and other higher-order effects may kick in, causing faster growth than quadratic. In other words, this growth condition shows that, on average, the Hessian tends to have a larger spectrum when evaluated further from the current point.

\begin{figure}[!htp]
  \centering
  \includegraphics[width=\linewidth]{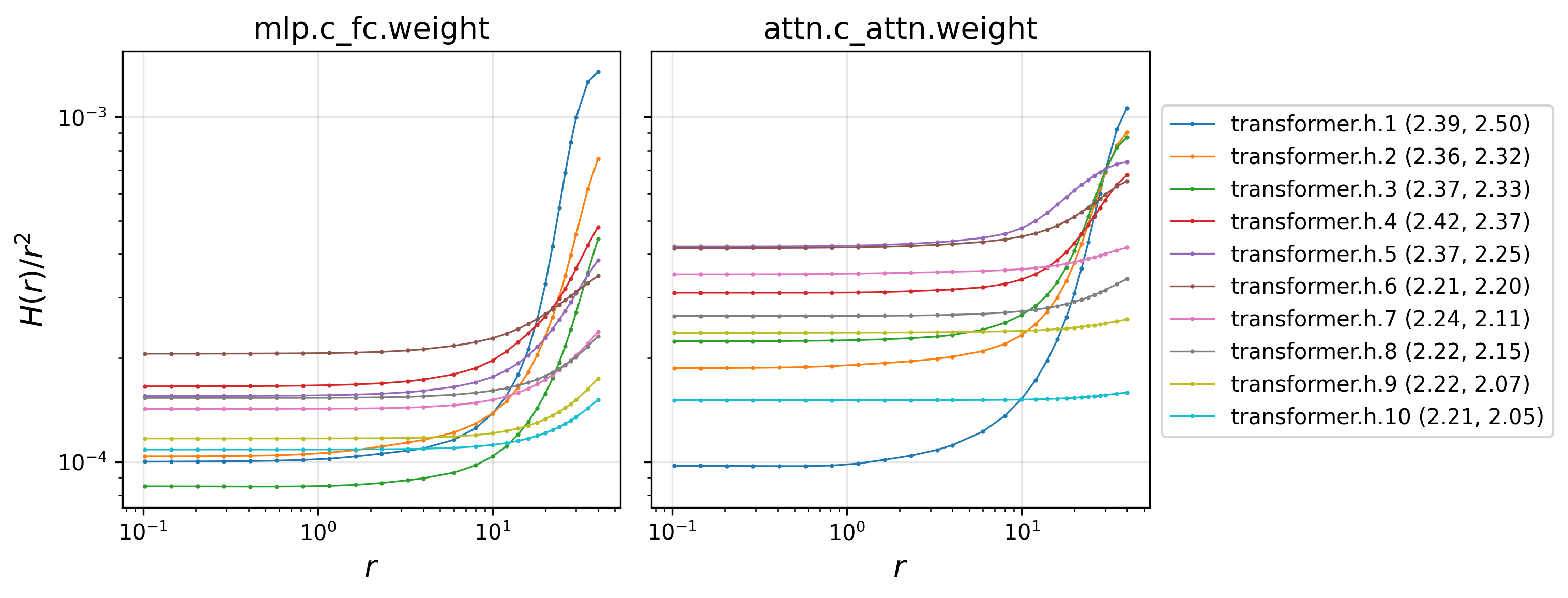}
  \caption{Numerical approximation of $H(r)/r^2$ on GPT-2 (small) \cite{radford2019language}, for fully connected layers (left) and attention layers (right). The values in parentheses in the legend indicate the estimated exponent $2 + \alpha$ of $H(r)$ starting from $r = 10^{0.5}$. For example, $2.39$ and $2.50$ in the first line are the exponents for the left and right panels, respectively. We sample $100$ random directions $\delta W$ from a standard normal distribution, sample the first $300$ examples from the C4-newslike validation split \cite{raffel2020exploring} and truncate each example to $100$ tokens. For each radius $r$, we construct perturbations $\dW u_i$, where $u_i$ is the input to the layer, and renormalize each $\dW u_i$ to have norm $r$. We then compute $H$ by subtracting $L(Wu_i) + \langle \nabla L(Wu_i), \dW u_i \rangle$ from $L(Wu_i + \dW u_i)$. This yields a total of $100 \times 300 \times (100-1) = 2{,}970{,}000$ remainder terms for approximating $H(r)/r^2$ per radius $r$. Note that the growth of $H$ will eventually plateau for very large $r$.}
  \label{fig:growth}
\end{figure}

\section{Analysis of the Model}
\label{sec:insights-into-muon}

We present three results in three subsections, based on progressively stronger assumptions. Proofs are provided in Section~\ref{sec:proof}.

\subsection{Alignment of Singular Spaces}
\label{sec:sing-space-unch}

The isotropic curvature model is rotationally invariant in the following sense: if the gradient $G$ is replaced by $O_1 G O_2$ for two orthogonal matrices $O_1$ and $O_2$, then the objective function
\[
-\Tr(Q G^\top) + \E_{\zeta \sim \text{sphere}} H(\|Q \zeta\|)
\]
remains invariant if the update matrix $Q$ is replaced by $O_1 Q O_2$.

This observation naturally suggests an alignment of singular spaces between the gradient matrix and the optimal update from the isotropic curvature model, as stated in the following proposition. Below, we only assume that the optimization program~\eqref{eq:our_model} admits a finite solution, which is implied, for example, if $H$ grows sufficiently fast. This condition is weaker than the super-quadratic growth assumption introduced in Section~\ref{sec:sh}.

\begin{proposition}\label{prop:orth_spaces}
There exists a global solution $Q^\star$ to the optimization program of the isotropic curvature model
\begin{equation}\nonumber
\min_{Q} -\Tr(Q G^\top) + \E_{\zeta \sim \textnormal{sphere}} H(\|Q \zeta\|)
\end{equation}
such that the singular spaces of $Q^\star$ and $G$ are aligned.
\end{proposition}

That is, there exist orthogonal matrices $U$ and $V$ and diagonal matrices with nonnegative entries $\Sigma$ and $\Sigma^\star$ such that $G = U\Sigma V^\top$ and $Q^\star = U\Sigma^\star V^\top$. This invariance of the row and column spaces would be lost if the update $Q$ were treated as a vector, neglecting its matrix structure. This result is consistent with the orthogonalization in Muon, which modifies only the spectrum of the gradient matrix while leaving both its left and right singular spaces unchanged.\footnote{Note that this is not the case for some matrix-gradient methods such as SOAP~\cite{vyas2024soap}.}

The statement of Proposition~\ref{prop:orth_spaces} uses ``there exists'' because, in the degenerate case where $G$ has repeated singular values, the singular vectors are defined only up to a unitary rotation within each degenerate block. In the \textit{generic} case, every global solution $Q^\star$ necessarily has singular spaces aligned with those of $G$.

\subsection{Spectrum Homogenization}
\label{sec:sh}

Building on the alignment of singular spaces, we now use the isotropic curvature model to analyze the spectrum of the optimal update, which provides insight into orthogonalized gradient methods.

The discussion in Section~\ref{sec:grow} and the empirical observations in Figure~\ref{fig:growth} suggest that $H(r) \propto r^{2 + \alpha}$ for $\alpha > 0$ over a certain range of $r$. Then $H(\sqrt{x}) \propto x^{1 + \alpha/2}$. Note that $x^{1 + \alpha/2}$ is a convex function of $x$. This motivates the following assumption to characterize the growth of the curvature function~$H$.

\begin{assumption}\label{ass:homo}
The curvature function $H$, defined on $[0, \infty)$, is strictly increasing, and the function $x \mapsto H(\sqrt{x})$ is convex in $x$.
\end{assumption}

Roughly speaking, this assumption corresponds to a super-quadratic growth of the curvature. Note that this assumption also implies that $H(r)$ is convex in $r$.

Any curvature function satisfying Assumption~\ref{ass:homo} grows sufficiently fast. This implies the assumption needed for Proposition~\ref{prop:orth_spaces}, which allows us to simultaneously unitarily diagonalize $Q^\star$ and $G$: let $G = U\Sigma V^\top$ and $Q^\star = U\Sigma^\star V^\top$ be the SVDs of $G$ and $Q^\star$, respectively. Write $\Sigma = \mathrm{diag}(\sigma_1, \ldots, \sigma_{\min\{m,n\}})$ and $\Sigma^\star = \mathrm{diag}(\sigma^\star_1, \ldots, \sigma^\star_{\min\{m,n\}})$. By definition, these singular values are all nonnegative.

\begin{theorem}\label{thm:homo}
Under Assumption~\ref{ass:homo}, there exists a global solution $Q^\star$ to the isotropic curvature model whose spectrum preserves the ordering of the spectrum of $G$ but is more homogeneous. That is, the following two statements hold:
\begin{itemize}
\item[(a)]
For any $1 \le i, j \le \min\{m, n\}$, if $\sigma_i \ge \sigma_j$, then $\sigma_i^\star \ge \sigma_j^\star$.
\item[(b)] For all $1 \le i, j \le \min\{m, n\}$, we have\footnote{By convention, we take $\frac00 = 1$, $\frac10 = \infty$, and $\infty \le \infty$.}
\begin{equation}\label{eq:homog}
\frac{\max\{\sigma^\star_i, \sigma^\star_j\}}{\min\{\sigma^\star_i, \sigma^\star_j\}} \le \frac{\max\{\sigma_i, \sigma_j\}}{\min\{\sigma_i, \sigma_j\}}.
\end{equation}
\end{itemize}
\end{theorem}

\begin{remark}
In fact, Part (a) follows from the proof of Proposition~\ref{prop:orth_spaces} in the generic case. However, we prefer to state it along with Part (b). Conventionally, singular values in an SVD are ordered by default (e.g., in descending order). Here, we do not impose such an ordering on the SVDs of $Q^\star$ and $G$; otherwise, Part (a) would be trivially implied.
\end{remark}

The second statement says that the singular values of the solution matrix $Q^\star$ are more uniform than those of the original gradient $G$. In this sense, the isotropic curvature model suggests that the optimal update direction should have singular values that are more homogeneous than those of the original gradient. We call this property, revealed by the theorem, \textit{spectrum homogenization}. Spectrum homogenization is scale-invariant; thus, this theorem does not provide information about the magnitudes of the singular values or specify a learning rate, which is often determined empirically.

Theorem \ref{thm:homo} implies that the singular values should ideally maintain the same ordering as those of the original gradient matrix. To this end, the update direction matrix should have its singular values obtained from those of the original gradient via a monotone transformation. However, this is typically not the case when implementing matrix-gradient methods with polynomial approximations, such as the Newton--Schultz iteration \cite{jordan2024muon}, which are generally not monotone. This suggests a potential direction for future research: developing efficient methods to solve the optimization problem in \eqref{eq:our_model} while preserving the ordering of singular values.

For illustration, consider $H(r) = c r^4$ for some $c > 0$, which satisfies Assumption~\ref{ass:homo}. The optimization program of the isotropic curvature model is equivalent to
\[
\min_{\tilde\sigma_1, \ldots, \tilde\sigma_k} -\sum_{i=1}^k \sigma_i \tilde\sigma_i + \frac{c}{n(n+2)} \left[ \left( \sum_{i=1}^k \tilde\sigma_i^2 \right)^2 + 2\sum_{i=1}^k \tilde\sigma_i^4 \right],
\]
where $k = \min\{m,n\}$. The optimal solution $(\sigma^\star_1, \ldots, \sigma^\star_k)$ must satisfy
\begin{equation}\label{eq:der_zero}
(\sigma_j^\star)^3 + D \sigma_j^\star - \frac{n(n+2)\sigma_j}{8c} = 0
\end{equation}
for all $j$, where $D = \frac{1}{2}\sum_{i=1}^k (\sigma_i^\star)^2$. It is straightforward to see that this system of cubic equations has a unique real root. Each component of this root is nonnegative, and it is positive if the corresponding $\sigma_j > 0$. It is also clear that if $\sigma_i > \sigma_j$, then we must have $\sigma^\star_i > \sigma^\star_j$. To see the second statement of Theorem~\ref{thm:homo}, assume $\sigma_i > \sigma_j > 0$. From \eqref{eq:der_zero}, it follows that
\[
\frac{(\sigma^\star_i)^3 + D \sigma^\star_i}{(\sigma^\star_j)^3 + D \sigma^\star_j} = \frac{\sigma_i}{\sigma_j}.
\]
Thus,
\[
\frac{\sigma^\star_i}{\sigma^\star_j} = \frac{(\sigma^\star_j)^2 + D}{(\sigma^\star_i)^2 + D} \cdot \frac{\sigma_i}{\sigma_j} < \frac{\sigma_i}{\sigma_j}.
\]

In contrast, if $H(r) = c r^2$, the solution $Q^\star$ to the isotropic curvature model is simply proportional to $G$.

\subsection{Orthogonalization}
\label{sec:orthogonalization}

Intuitively, the most extreme form of spectrum homogenization is orthogonalization, which makes all singular values equal. To study when this occurs, we next consider a scenario where the curvature function $H$ increases very rapidly. This means its derivative $H'(r)$ increases sharply from a small value to a large value over a short range. In the limit, we consider this range to be zero, creating a jump discontinuity. We assume there is a \textit{kink} at some radius $\tilde r$, such that $H'(\tilde r-)$ is small while $H'(\tilde r+)$ is large. Our next result shows that under this extreme growth condition, the optimal solution to the isotropic curvature model is orthogonalization.

To formally state our result, we consider the following assumption.

\begin{assumption}\label{ass:orth}
The curvature function $H$ is a non-decreasing convex function such that there exists $\tilde r > 0$ where the left-derivative is $H'(\tilde r-) = A$ and the right-derivative is $H'(\tilde r+) = B$ for some $0 \le A < B < \infty$.
\end{assumption}

This assumption states that the subdifferential of $H$ at $\tilde r$ is the interval $[A, B]$. Throughout this subsection, we consider $m \ge n$, meaning $Q$ has at least as many rows as columns. We assume that the gradient matrix $G$ is of full rank; that is, $\sigma_{\min\{m,n\}} = \sigma_n > 0$.

\begin{theorem}\label{thm:orth}
Let $H$ satisfy Assumption~\ref{ass:orth} with a sufficiently small $A$ and a sufficiently large $B$. Then there exists an optimal solution $Q^\star$ to the isotropic curvature model that takes the form $Q^\star = c \cdot \msgn(G) \equiv c U V^\top$ for some scalar $c > 0$, where $U$ and $V$ are from the SVD of $G$.
\end{theorem}

This theorem provides a justification for the orthogonalization in Muon and its extensions like PolarGrad~\cite{lau2025polargrad}, which can be understood as operating under an implicit assumption that the curvature function grows slowly at first and then suddenly ``takes off.'' This theorem can be viewed as a limiting counterpart to Theorem~\ref{thm:homo}, as orthogonalization is the most extreme form of homogenization.

The intuition behind this theorem can be gleaned from its proof in Section~\ref{sec:proof}. The first-order optimality condition requires that $G$ be an element of the subdifferential of the term $\E_{\zeta \sim \text{sphere}} H(\|Q \zeta\|)$. As revealed by the proof, the scalar $c$ in Theorem~\ref{thm:orth} is $\tilde r$. When $r < \tilde r$, the derivative $H'(r)$ is small, and when $r > \tilde r$, it is large. The first-order condition can only be satisfied if the wide interval of the subdifferential at $\tilde r$ is leveraged. The stationary point must therefore occur for a $Q^\star$ such that $\|Q^\star \zeta\| = \tilde r$ almost surely for $\zeta$ on the unit sphere, which implies that $Q^\star$ must have scaled orthonormal columns. This is why Theorem~\ref{thm:orth} (and Proposition~\ref{prop:orth_converse} below) requires the technical condition $m \ge n$. However, as we will remark in the proof in Section~\ref{sec:proof}, this condition could be dropped, albeit at the cost of more complicated mathematical statements.

The kink condition in Assumption~\ref{ass:orth} is also necessary for gradient orthogonalization, as shown in the next result.

\begin{proposition}\label{prop:orth_converse}
Let $H$ satisfy Assumption~\ref{ass:homo} and let $G$ not be a (scaled) orthonormal matrix; that is, its singular values are not all equal. If the isotropic curvature model has a global solution that takes the form $Q^\star = c \cdot \msgn(G)$ for some scalar $c > 0$, then $H$ must have a kink as described in Assumption~\ref{ass:orth}.
\end{proposition}

However, this interpretation of Theorem~\ref{thm:orth} and Proposition~\ref{prop:orth_converse} requires caution. While our numerical results suggest that $H$ grows quickly, it is unlikely to exhibit the extreme case of a derivative jumping suddenly from a small value to a large one. In light of this, the gradient orthogonalization in Muon could be interpreted as a reasonable approach from an extreme-case perspective.

In short, our analysis suggests that the optimal strategy is generally not full orthogonalization, but rather homogenization to some degree. This is consistent with numerical experiments showing that rough and exact orthogonalization yield similar performance in pretraining large language models~\cite{kimi2025k2}. Determining the optimal level of homogenization presents an opportunity for future research.

\section{Proofs}
\label{sec:proof}

\subsection{Proofs for Sections~\ref{sec:sing-space-unch} and~\ref{sec:sh}}
\label{sec:proofs-sect-refs}

The proof of Proposition~\ref{prop:orth_spaces} immediately follows from von Neumann's trace inequality.

\begin{lemma}[von Neumann's trace inequality]\label{lemma:orth}
Let $A = U_1\Sigma_1 V_1^\top$ and $B = U_2\Sigma_2 V_2^\top$ be the SVDs of two matrices of the same size. Then
\[
|\Tr(AB^\top)| \le \Tr(\Sigma_1 \Sigma_2^\top).
\]
\end{lemma}

Note that we assume the default SVD where the singular values in both $\Sigma_1$ and $\Sigma_2$ are in decreasing order.

\begin{proof}[Proof of Proposition~\ref{prop:orth_spaces}]

To see how Proposition~\ref{prop:orth_spaces} follows, note that the second term in the objective, $\E_{\zeta \sim \textnormal{sphere}} H(\|Q \zeta\|)$, does not change if one varies the singular spaces of $Q$, as long as the spectrum of $Q$ is fixed.

Let $Q^\star$ be any global solution to the isotropic curvature model and write $Q^\star = U_Q \Sigma^\star V^\top_Q$ for its SVD. Let $G = U \Sigma V^\top$ be the SVD of the gradient matrix. By the optimality of $Q^\star$, we know that $Q^\star$ maximizes $\Tr(Q G^\top)$ while fixing its singular values, $\Sigma^\star$. By Lemma~\ref{lemma:orth},
\[
\Tr(Q^\star G^\top) \le \Tr(\Sigma^\star \Sigma^\top),
\]
and equality holds if $U_Q = U$ and $V_Q = V$. This implies that the matrix $\tilde{Q} = U \Sigma^\star V^\top$, which has the same singular values as $Q^\star$, yields an objective value less than or equal to that of $Q^\star$. Thus, $U \Sigma^\star V^\top$ must also be a global solution, and it has singular spaces aligned with those of $G$.

\end{proof}

\begin{remark}
The equality in von Neumann's trace inequality holds if and only if the singular spaces are exactly aligned in the generic case where all singular values are distinct. In this case, the proof of Proposition~\ref{prop:orth_spaces} reveals that the singular values of $G$ and $Q^\star$ have the same ordering.

\end{remark}

Next, we prove Theorem~\ref{thm:homo}.

\begin{proof}[Proof of Theorem~\ref{thm:homo}]

As noted in the proof of Proposition~\ref{prop:orth_spaces}, Part (a) follows from aligning the singular spaces in the generic case. Now we prove Part (b).

For simplicity, write $q(x) = H(\sqrt{x})$, which is convex by assumption. The optimization problem becomes
\[
\min_Q -\Tr(Q G^\top) + \E_{\zeta} q(\|Q \zeta\|^2),
\]
where we omit $\sim \text{sphere}$ in the subscript for brevity.

From Proposition~\ref{prop:orth_spaces}, we can assume without loss of generality that both $G$ and the optimal solution $Q^\star$ are diagonal. Let $Q = \mathrm{diag}(\tilde\sigma_1, \ldots, \tilde\sigma_k)$, where $k = \min\{m,n\}$. Then the program becomes
\[
\min_{\tilde\sigma_1, \ldots, \tilde\sigma_k} -\sum_{i=1}^k \sigma_i \tilde\sigma_i + \E_{\zeta} q(\tilde\sigma_1^2\zeta_1^2 + \cdots + \tilde\sigma_k^2 \zeta_k^2).
\]
If $\sigma_i = 0$ for some $i$, then since $q$ is strictly increasing, the optimal value $\sigma_i^\star$ must be 0 as well. This immediately shows that \eqref{eq:homog} holds whenever one of $\sigma_i$ and $\sigma_j$ is 0.

Now we consider the non-degenerate case where $\sigma_i\sigma_j \ne 0$. Let $(\sigma_1^\star, \ldots, \sigma_k^\star)$ be a minimizer, which must exist since $q$ is convex and strictly increasing. Then the first-order optimality conditions are
\[
2\E_{\zeta} \left[ \partial q\left(\sum_{l=1}^k (\sigma_l^\star)^2\zeta_l^2\right)\zeta_i^2 \sigma_i^\star \right] = \sigma_i
\]
for each $i=1, \ldots, k$. Comparing the conditions for indices $i$ and $j$:
\begin{equation}\label{eq:twosig}
2\E_{\zeta} \left[ \partial q\left(\sum_{l=1}^k (\sigma_l^\star)^2\zeta_l^2\right)\zeta_i^2 \sigma_i^\star \right] = \sigma_i, \quad 2\E_{\zeta} \left[ \partial q\left(\sum_{l=1}^k (\sigma_l^\star)^2\zeta_l^2\right)\zeta_j^2 \sigma_j^\star \right] = \sigma_j.
\end{equation}

Without loss of generality, assume $\sigma_j > \sigma_i$. By Part (a), this implies $\sigma_j^\star \ge \sigma_i^\star$. However, strict inequality must hold ($\sigma_j^\star > \sigma_i^\star$), because if $\sigma_j^\star = \sigma_i^\star$, the first-order conditions \eqref{eq:twosig} would imply $\sigma_j = \sigma_i$, which contradicts our assumption. To prove homogenization, we need to show $\sigma_j^\star / \sigma_i^\star \le \sigma_j / \sigma_i$, which is equivalent to $\sigma_j^\star \sigma_i - \sigma_i^\star \sigma_j \le 0$.
We examine the sign of this difference:
\[
\begin{aligned}
& \sigma_j^\star \sigma_i - \sigma_i^\star \sigma_j \\
& = \sigma_j^\star \cdot 2\E_{\zeta} \left[ \partial q\left(\sum_{l=1}^k (\sigma_l^\star)^2\zeta_l^2\right)\zeta_i^2 \sigma_i^\star \right] - \sigma_i^\star \cdot 2\E_{\zeta} \left[ \partial q\left(\sum_{l=1}^k (\sigma_l^\star)^2\zeta_l^2\right)\zeta_j^2 \sigma_j^\star \right] \\
& = 2 \sigma_i^\star \sigma_j^\star \E_{\zeta} \left[ \partial q\left(\sum_{l=1}^k (\sigma_l^\star)^2 \zeta_l^2 \right)\zeta_i^2 - \partial q\left(\sum_{l=1}^k (\sigma_l^\star)^2 \zeta_l^2 \right)\zeta_j^2 \right]\\
& = 2 \sigma_i^\star \sigma_j^\star \E_{\zeta}(\zeta_i^2 - \zeta_j^2) \partial q\left(\sum_{l=1}^k (\sigma_l^\star)^2 \zeta_l^2 \right) .
\end{aligned}
\]

The proof reduces to showing that the expectation is negative. Let $S(\zeta) = \sum_{l \ne i, j} (\sigma_l^\star)^2 \zeta_l^2$. Due to symmetry between $\zeta_i$ and $\zeta_j$,
\[
\E_{\zeta} (\zeta_i^2 - \zeta_j^2) \partial q\left(S(\zeta) + (\sigma_i^\star)^2 \zeta_i^2+ (\sigma_j^\star)^2 \zeta_j^2 \right) = \E_{\zeta} (\zeta_j^2 - \zeta_i^2) \partial q\left(S(\zeta) + (\sigma_i^\star)^2 \zeta_j^2+ (\sigma_j^\star)^2 \zeta_i^2 \right)
\]
Hence, it suffices to show that the expectation below is negative:
\[
(\zeta_i^2 - \zeta_j^2) \left[ \partial q\left(S(\zeta) + (\sigma_i^\star)^2 \zeta_i^2 + (\sigma_j^\star)^2 \zeta_j^2 \right) - \partial q\left(S(\zeta) + (\sigma_i^\star)^2 \zeta_j^2 + (\sigma_j^\star)^2 \zeta_i^2 \right) \right].
\]
Let $A = S(\zeta) + (\sigma_i^\star)^2 \zeta_i^2 + (\sigma_j^\star)^2 \zeta_j^2$ and $B = S(\zeta) + (\sigma_i^\star)^2 \zeta_j^2 + (\sigma_j^\star)^2 \zeta_i^2$.
Then $B-A = ((\sigma_j^\star)^2 - (\sigma_i^\star)^2)(\zeta_i^2 - \zeta_j^2)$.
Since we assumed $\sigma_j^\star > \sigma_i^\star$, the sign of $B-A$ is the same as the sign of $\zeta_i^2 - \zeta_j^2$.
By the convexity of $q$, its derivative $\partial q$ is non-decreasing, so $\partial q(B) - \partial q(A)$ also has the same sign as $B-A$.
Thus, $(\zeta_i^2 - \zeta_j^2)(\partial q(A) - \partial q(B)) \le 0$.
The inequality is strict unless $\zeta_i^2 = \zeta_j^2$ or the derivative is constant. After taking the expectation over $\zeta$, the strict inequality holds, which completes the proof.

\end{proof}

\begin{remark}
The convexity of $H(\sqrt{x})$, which ensures super-quadratic growth, can be slightly relaxed. The theorem still holds if $H(r)$ grows slowly for small $r$, such as $H(r) = O(r^2)$, provided $H(\sqrt{x})$ is convex for sufficiently large $x$.
\end{remark}

\subsection{Proofs for Section~\ref{sec:orthogonalization}}
\label{sec:proofs-section-ref}

We first state a lemma before proving Theorem~\ref{thm:orth}. Recall that $\zeta = (\zeta_1, \ldots, \zeta_n)$ is uniformly sampled from the unit sphere in $\R^n$.
\begin{lemma}\label{lm:eta_exist}
For any positive constants $C_1, \ldots, C_n$, there exist positive constants $a < b$ and a random variable $\eta$ satisfying $a \le \eta \le b$ almost surely such that
\[
\E \eta \zeta_i^2 = C_i
\]
for all $i = 1, \ldots, n$.

\end{lemma}

\begin{proof}[Proof of Lemma~\ref{lm:eta_exist}]
Consider the set of all random variables $\eta$ that are bounded away from 0 and $\infty$. Define a linear map $\mathcal{E}$ from this set to $\mathbb{R}^n$ as 
\[
\mathcal{E}(\eta) = (\E \eta \zeta_1^2, \ldots, \E\eta \zeta_n^2).
\] 
The image of this map, denoted by $K$, is the set of all achievable expectation vectors. Clearly, $K$ is a convex cone in the positive orthant $\R^n_{++} := \{x \in \R^n: x_i > 0 \text{ for all } i\}$. If $K$ is the entire positive orthant, the proof is complete.

Now we prove by contradiction by assuming $K$ is not the entire positive orthant. Since $K$ is convex, the interior of its closure $\bar{K}$ is the same as the interior of $K$, which is a proper subset of $\mathbb{R}^n_{++}$ by assumption. Hence, the interior of $\bar{K}$ is a proper subset of $\mathbb{R}^n_{++}$. Hence, there exists a point $(C_1, \ldots, C_n) \in \mathbb{R}^n_{++}$ that is not in $\bar{K}$.

Next, we apply the separating hyperplane theorem to $\{C\}$, which is closed, convex, and compact, and $\bar{K}$, which is closed and convex. Then there must exist a vector $v$ such that 
\[
v^\top C < c, \quad v^\top w \ge c
\]
for all $w \in \bar{K}$ and some constant $c$. Since $\bar{K}$ is a cone, the hyperplane can be chosen to pass through the origin, meaning we can set $c = 0$. Thus, there exists a nonzero vector $v$ such that 
\[
v^\top w \ge 0
\]
for all $w \in K$, and meanwhile, $v^\top C = v_1 C_1 + \cdots + v_n C_n < 0$. Since $C_1, \ldots, C_n$ are all positive, there is at least one component of $v$ that is negative, say, $v_{i_0} < 0$. 

To finish the proof, let $w \in K$ take the form $w = \mathcal{E}(\eta) = (\mathbb{E} \eta \zeta^2_1, \ldots, \mathbb{E}\eta \zeta^2_n)$ for any $\eta$ that is bounded away both from 0 and $\infty$. Then we get
\[
v^\top w = v_1\mathbb{E} \eta \zeta^2_1 + \cdots + v_n\mathbb{E}\eta \zeta^2_n = \mathbb{E}\left[\eta(v_1\zeta^2_1 + \cdots + v_n \zeta^2_n)   \right] \ge 0
\]
for all $\eta$. However, this is clearly impossible as we can let $\eta$ be very large when $\zeta_{i_0}$ is close to 1 or $-1$ and otherwise set $\eta$ to be very small. This yields a contradiction, thereby proving that $K$ must be the entire positive orthant.

\end{proof}

\begin{proof}[Proof of Theorem~\ref{thm:orth}]
Since the optimization program~\eqref{eq:our_model} is convex, we only need to verify that $Q^\star = \tilde r \msgn(G)$ satisfies the first-order condition
\begin{equation}\label{eq:1st_cond}
G \in \E_{\zeta} \left[ \partial_{Q} H(\|Q^\star \zeta\|) \right].
\end{equation}
Note that for all $\zeta$ from the unit sphere, $\|Q^\star \zeta\| = \tilde r$ since $m \ge n$ and the columns of $\msgn(G)$ are orthonormal. By the chain rule for subdifferentials,
\[
\begin{aligned}
\partial_{Q} H(\|Q^\star \zeta\|) &= \partial H(\tilde r) \frac{Q^\star \zeta \zeta^\top}{\|Q^\star \zeta\|} \\
&= \partial H(\tilde r) \frac{\tilde r \msgn(G) \zeta \zeta^\top}{\tilde r} \\
&= \partial H(\tilde r) U V^\top \zeta \zeta^\top,
\end{aligned}
\]
where $\partial H(\tilde r)$ can take any value from the subdifferential $[A, B]$.

To satisfy \eqref{eq:1st_cond}, it is sufficient to prove that there exists a random variable $\eta(\zeta)$ taking values in $[A, B]$ almost surely such that
\[
\E_{\zeta} [\eta(\zeta) U V^\top \zeta \zeta^\top] = G = U\Sigma V^\top,
\]
which is equivalent to
\[
U \left( \E_{\zeta} [\eta(\zeta) (V^\top\zeta) (V^\top\zeta)^\top] \right) V^\top = U\Sigma V^\top.
\]
Since $V^\top\zeta$ has the same distribution as $\zeta$, the proof is complete if we can show there exists such an $\eta$ satisfying
\begin{equation}\nonumber
\E_{\zeta} [\eta(\zeta) \zeta \zeta^\top] = \Sigma.
\end{equation}
The off-diagonal terms are zero by symmetry. For the diagonal terms, we need $\E[\eta(\zeta) \zeta_i^2] = \sigma_i$ for all $i$. For a sufficiently small $A$ and large $B$, the convex set of achievable vectors $\{(\E \eta \zeta_1^2, \ldots, \E \eta \zeta_n^2) : A \le \eta \le B\}$ contains the vector $(\sigma_1, \ldots, \sigma_n)$, a consequence of Lemma~\ref{lm:eta_exist}. This completes the proof.

\end{proof}

\begin{remark}

The practical significance lies in the finiteness of $B$, which roughly translates to the fact that gradient orthogonalization is optimal without requiring an infinite subdifferential. However, the proof does not specify how large $A$ and $B$ need to be for the kink to satisfy the first-order condition since it is not constructive. Intuitively, it depends on the dimensions and the range of the spectrum of $G$. It is of practical interest to obtain sharp bounds on $A$ and $B$ for Theorem~\ref{thm:orth}.

\end{remark}

Next, we turn to the proof of Proposition~\ref{prop:orth_converse}.

\begin{proof}[Proof of Proposition~\ref{prop:orth_converse}]
Since $Q^\star = c \cdot \msgn(G)$ is an optimal solution to the convex program, it must satisfy the first-order condition
\begin{equation}\label{eq:1st_cond2}
G \in \E_{\zeta} \left[ \partial_{Q} H(\|Q^\star \zeta\|) \right].
\end{equation}
As in the proof of Theorem~\ref{thm:orth}, we have $\|Q^\star\zeta\| = c$ for all $\zeta$. The condition becomes
\[
G \in \E_{\zeta} \left[ \partial H(c) U V^\top \zeta \zeta^\top \right].
\]
If $H$ is differentiable at $c$, then the subdifferential $\partial H(c)$ is the singleton $\{H'(c)\}$, and we get
\[
G = H'(c) U V^\top \E_{\zeta} [\zeta \zeta^\top].
\]
Since $\E[\zeta \zeta^\top] = \frac{1}{n} I$, this simplifies to
\[
U\Sigma V^\top = \frac{H'(c)}{n} U V^\top.
\]
This implies $\Sigma = \frac{H'(c)}{n} I$, meaning all singular values of $G$ must be equal. This contradicts the assumption that $G$ is not a scaled orthogonal matrix. Therefore, $H$ cannot be differentiable at $c$ and must have a kink.
\end{proof}

\begin{remark}
Both Theorem~\ref{thm:orth} and Proposition~\ref{prop:orth_converse} assume $m \ge n$. This condition ensures that $\|Q\zeta\|$ is a constant for all unit vectors $\zeta$ when the columns of $Q \in \R^{m \times n}$ are scaled orthonormal. The two results can be extended to the case $m < n$, but the statements might involve approximations and might not be as precise. Roughly speaking, the proof idea should continue to work by recognizing that $Q$ preserves the norm in a subspace of dimension $m$. Moreover, for sufficiently large $m$, concentration of measure ensures that the first $m$ components of $\zeta$ are approximately sampled from a sphere in $\R^m$. We leave the detailed proofs for this extension to future work.
\end{remark}


\section{Discussion}
\label{sec:discussion}

In this paper, we investigate how the gradient orthogonalization used in Muon and other related methods contributes to their superior performance in training large-scale deep learning systems such as large language models. Our approach is to propose a convex program, which we call the isotropic curvature model, to understand how optimization proceeds in a single iteration. A key ingredient in the development of this model is the Taylor expansion of the per-sample loss function instead of the total objective, thereby capturing the matrix structure of the weights. Our model shows that under a super-quadratic growth condition on the high-order terms of the loss function, the optimal update matrix shares the same singular spaces as the original gradient matrix, but its singular values should be more homogeneous. This result immediately clarifies why treating weights as matrices is not only conceptually sound but also technically beneficial. In addition, we identify the condition under which gradient orthogonalization becomes optimal: when the high-order terms exhibit a ``kink'' where their growth rate sharply increases. This condition can be viewed as an extreme-case approximation, suggesting that while gradient orthogonalization is beneficial, it may not fully exploit the potential of matrix-gradient methods.

At a high level, the rationale behind the isotropic curvature model, supported by empirical evidence, points to a crucial feature of deep learning optimization. Conventionally, first- and second-order information are considered; these can in general take arbitrary values, and preconditioning methods generally require the evaluation of both.\footnote{Note that there are exceptions. For example, Nesterov's accelerated gradient method extracts certain second-order information from first-order gradients.} However, in deep learning optimization, while first-order information (the gradient) is arbitrary and must be obtained via differentiation, it seems plausible that high-order information possesses a structure that can be leveraged without explicitly computing the Hessian. This may be due to the immense scale of systems like large language models and the isotropic properties underlying their architectures. If this hypothesis holds true, it could pave the way for a large class of \textit{auto-preconditioned} methods. In fact, both Muon and Adam (which can be seen as a smoothed version of signed stochastic gradient descent \cite{bernstein2018signsgd}) are examples of auto-preconditioned methods, as their conditioning of first-order gradients does not explicitly leverage any second-order information. It is likely that more novel methods will emerge from exploring more fine-grained auto-preconditioners.

From a technical perspective, the isotropic curvature model is currently phenomenological, and future research is needed to refine it and unveil its underlying mechanisms. First, rigorously justifying the isotropy assumption would be of great interest, and our current assumptions may require modification to be made mathematically precise. Work characterizing the geometric properties of the optimization landscape for deep learning could provide valuable insights \cite{neyshabur2015path,fang2021exploring,wang2025muon}. Another direction is to unveil the mechanism for the super-quadratic growth of high-order terms, which might be illuminated by connecting it to phenomena like the edge of stability \cite{wu2018sgd,cohengradient,cohenunderstanding} and the Hessian structure of neural networks \cite{zhang2024transformers,kunstner2024heavy}. More broadly, the isotropic curvature model and its future extensions could potentially be used in conjunction with the modular interpretation of optimization \cite{bernstein2024old,bernstein2024modular}. It would also be interesting to investigate whether the model could be extended from a single iteration to multiple iterations for convergence analysis. To better model practice, future research should also consider a noisy gradient $G$ from a mini-batch, since it is empirically observed that the relative performance of Muon compared to Adam depends on batch size \cite{essentialai2025practical}. In the same spirit, incorporating momentum via exponential moving average into the model is another important direction.

From a practical standpoint, the isotropic curvature model suggests a new approach to designing optimizers for large language models. This approach would begin by approximating the curvature function $H$ for the target neural network architecture, followed by solving the optimization program of the isotropic curvature model to find the update matrix. The first step requires care, as $H$ may vary across different layers and over the course of training. This variability might explain why the relative empirical performance of Muon compared to Adam varies over problem instances \cite{semenov2025benchmarking,wen2025fantastic}. Given a curvature function $H$, a key problem is whether the optimal update matrix can be found efficiently using methods suitable for GPU environments, such as the Newton--Schulz iteration. An additional desired property, suggested by Theorem \ref{thm:homo}, is that the transformation of singular values be monotone.

\subsection*{Acknowledgments}

I am grateful to Tim Lau for introducing me to Muon in early 2025. I would like to thank Xuyang Chen, Alex Damian, Shengtao Guo, Nathan Srebro, and Zihan Zhu for helpful discussions. This research was supported in part by NSF grant DMS-2310679 and Wharton AI for Business.

\bibliographystyle{abbrv}
\bibliography{ref}

\begin{thebibliography}{10}

\bibitem{ahn2025dion}
K.~Ahn, B.~Xu, N.~Abreu, Y.~Fan, G.~Magakyan, P.~Sharma, Z.~Zhan, and
  J.~Langford.
\newblock Dion: {D}istributed orthonormalized updates.
\newblock {\em arXiv preprint arXiv:2504.05295}, 2025.

\bibitem{amsel2025polar}
N.~Amsel, D.~Persson, C.~Musco, and R.~Gower.
\newblock The {Polar Express}: {O}ptimal matrix sign methods and their
  application to the {Muon} algorithm.
\newblock {\em arXiv preprint arXiv:2505.16932}, 2025.

\bibitem{an2025asgo}
K.~An, Y.~Liu, R.~Pan, S.~Ma, D.~Goldfarb, and T.~Zhang.
\newblock {ASGO}: Adaptive structured gradient optimization.
\newblock {\em arXiv preprint arXiv:2503.20762}, 2025.

\bibitem{bernstein2024old}
J.~Bernstein and L.~Newhouse.
\newblock Old optimizer, new norm: An anthology.
\newblock In {\em OPT 2024: Optimization for Machine Learning}, 2024.

\bibitem{bernstein2024modular}
J.~Bernstein and L.~Newhouse.
\newblock Modular duality in deep learning.
\newblock In {\em Proceedings of the International Conference on Machine
  Learning (ICML)}, 2025.

\bibitem{bernstein2018signsgd}
J.~Bernstein, Y.-X. Wang, K.~Azizzadenesheli, and A.~Anandkumar.
\newblock sign{SGD}: {C}ompressed optimisation for non-convex problems.
\newblock In {\em Proceedings of the International Conference on Machine
  Learning (ICML)}, 2018.

\bibitem{carlson2015stochasticRBM}
D.~Carlson, V.~Cevher, and L.~Carin.
\newblock Stochastic spectral descent for restricted {B}oltzmann machines.
\newblock In {\em Proceedings of the International Conference on Artificial
  Intelligence and Statistics (AISTATS)}, 2015.

\bibitem{carlson2015preconditioned}
D.~Carlson, E.~Collins, Y.-P. Hsieh, L.~Carin, and V.~Cevher.
\newblock Preconditioned spectral descent for deep learning.
\newblock In {\em Advances in Neural Information Processing Systems (NeurIPS)},
  2015.

\bibitem{chang2025convergence}
D.~Chang, Y.~Liu, and G.~Yuan.
\newblock On the convergence of muon and beyond.
\newblock {\em arXiv preprint arXiv:2509.15816}, 2025.

\bibitem{chen2025muon}
L.~Chen, J.~Li, and Q.~Liu.
\newblock Muon optimizes under spectral norm constraints.
\newblock {\em arXiv preprint arXiv:2506.15054}, 2025.

\bibitem{cohenunderstanding}
J.~Cohen, A.~Damian, A.~Talwalkar, J.~Z. Kolter, and J.~D. Lee.
\newblock Understanding optimization in deep learning with central flows.
\newblock In {\em The Thirteenth International Conference on Learning
  Representations}, 2025.

\bibitem{cohengradient}
J.~Cohen, S.~Kaur, Y.~Li, J.~Z. Kolter, and A.~Talwalkar.
\newblock Gradient descent on neural networks typically occurs at the edge of
  stability.
\newblock In {\em International Conference on Learning Representations}, 2021.

\bibitem{crawshaw2025exploration}
M.~Crawshaw, C.~Modi, M.~Liu, and R.~M. Gower.
\newblock An exploration of non-{E}uclidean gradient descent: {M}uon and its
  many variants.
\newblock {\em arXiv preprint arXiv:2510.09827}, 2025.

\bibitem{fan2025implicit}
C.~Fan, M.~Schmidt, and C.~Thrampoulidis.
\newblock Implicit bias of spectral descent and {M}uon on multiclass separable
  data.
\newblock {\em arXiv preprint arXiv:2502.04664}, 2025.

\bibitem{fang2021exploring}
C.~Fang, H.~He, Q.~Long, and W.~J. Su.
\newblock Exploring deep neural networks via layer-peeled model: {M}inority
  collapse in imbalanced training.
\newblock {\em Proceedings of the National Academy of Sciences},
  118(43):e2103091118, 2021.

\bibitem{gruntkowska2025drop}
K.~Gruntkowska, Y.~Maziane, Z.~Qu, and P.~Richt{\'a}rik.
\newblock Drop-{M}uon: {U}pdate less, converge faster.
\newblock {\em arXiv preprint arXiv:2510.02239}, 2025.

\bibitem{gupta2018shampoo}
V.~Gupta, T.~Koren, and Y.~Singer.
\newblock Shampoo: {P}reconditioned stochastic tensor optimization.
\newblock In {\em Proceedings of the International Conference on Machine
  Learning (ICML)}, 2018.

\bibitem{he2025low}
C.~He, Z.~Deng, and Z.~Lu.
\newblock Low-rank orthogonalization for large-scale matrix optimization with
  applications to foundation model training.
\newblock {\em arXiv preprint arXiv:2509.11983}, 2025.

\bibitem{he2025demuon}
C.~He, S.~Ren, J.~Mao, and E.~G. Larsson.
\newblock De{M}uon: {A} decentralized {M}uon for matrix optimization over
  graphs.
\newblock {\em arXiv preprint arXiv:2510.01377}, 2025.

\bibitem{huang2025limuon}
F.~Huang, Y.~Luo, and S.~Chen.
\newblock Li{M}uon: {L}ight and fast {M}uon optimizer for large models.
\newblock {\em arXiv preprint arXiv:2509.14562}, 2025.

\bibitem{jordan2024muon}
K.~Jordan, Y.~Jin, V.~Boza, Y.~Jiacheng, F.~Cecista, L.~Newhouse, and
  J.~Bernstein.
\newblock Muon: {A}n optimizer for hidden layers in neural networks, 2024.

\bibitem{khaled2025muonbp}
A.~Khaled, K.~Ozkara, T.~Yu, M.~Hong, and Y.~Park.
\newblock Muon{BP}: {F}aster {M}uon via block-periodic orthogonalization.
\newblock {\em arXiv preprint arXiv:2510.16981}, 2025.

\bibitem{kingma2015}
D.~P. Kingma and J.~L. Ba.
\newblock Adam: {A} method for stochastic optimization.
\newblock In {\em International Conference on Learning Representations (ICLR)},
  2015.

\bibitem{kovalev2025understanding}
D.~Kovalev.
\newblock Understanding gradient orthogonalization for deep learning via
  non-{E}uclidean trust-region optimization.
\newblock {\em arXiv preprint arXiv:2503.12645}, 2025.

\bibitem{kunstner2024heavy}
F.~Kunstner, R.~Yadav, A.~Milligan, M.~Schmidt, and A.~Bietti.
\newblock Heavy-tailed class imbalance and why {A}dam outperforms gradient
  descent on language models.
\newblock {\em arXiv preprint arXiv:2402.19449}, 2024.

\bibitem{lau2025polargrad}
T.~T.-K. Lau, Q.~Long, and W.~Su.
\newblock Polar{G}rad: {A} class of matrix-gradient optimizers from a unifying
  preconditioning perspective.
\newblock {\em arXiv preprint arXiv:2505.21799}, 2025.

\bibitem{li2025muon}
J.~Li and M.~Hong.
\newblock A note on the convergence of {Muon}.
\newblock {\em arXiv preprint arXiv:2502.02900}, 2025.

\bibitem{li2025normuon}
Z.~Li, L.~Liu, C.~Liang, W.~Chen, and T.~Zhao.
\newblock Nor{M}uon: {M}aking {M}uon more efficient and scalable.
\newblock {\em arXiv preprint arXiv:2510.05491}, 2025.

\bibitem{liu2025muon}
J.~Liu, J.~Su, X.~Yao, Z.~Jiang, G.~Lai, Y.~Du, Y.~Qin, W.~Xu, E.~Lu, J.~Yan,
  Y.~Chen, H.~Zheng, Y.~Liu, S.~Liu, B.~Yin, W.~He, H.~Zhu, Y.~Wang, J.~Wang,
  M.~Dong, Z.~Zhang, Y.~Kang, H.~Zhang, X.~Xu, Y.~Zhang, Y.~Wu, X.~Zhou, and
  Z.~Yang.
\newblock Muon is scalable for {LLM} training.
\newblock {\em arXiv preprint arXiv:2502.16982}, 2025.

\bibitem{loshchilov2019decoupled}
I.~Loshchilov and F.~Hutter.
\newblock Decoupled weight decay regularization.
\newblock In {\em International Conference on Learning Representations (ICLR)},
  2019.

\bibitem{neyshabur2015path}
B.~Neyshabur, R.~R. Salakhutdinov, and N.~Srebro.
\newblock Path-{SGD}: {P}ath-normalized optimization in deep neural networks.
\newblock {\em Advances in neural information processing systems}, 28, 2015.

\bibitem{pethick2025training}
T.~Pethick, W.~Xie, K.~Antonakopoulos, Z.~Zhu, A.~Silveti-Falls, and V.~Cevher.
\newblock Training deep learning models with norm-constrained {LMOs}.
\newblock In {\em Proceedings of the International Conference on Machine
  Learning (ICML)}, 2025.

\bibitem{radford2019language}
A.~Radford, J.~Wu, R.~Child, D.~Luan, D.~Amodei, and I.~Sutskever.
\newblock Language models are unsupervised multitask learners.
\newblock {\em OpenAI blog}, 2019.

\bibitem{raffel2020exploring}
C.~Raffel, N.~Shazeer, A.~Roberts, K.~Lee, S.~Narang, M.~Matena, Y.~Zhou,
  W.~Li, and P.~J. Liu.
\newblock Exploring the limits of transfer learning with a unified text-to-text
  transformer.
\newblock {\em Journal of Machine Learning Research}, 21(140):1--67, 2020.

\bibitem{riabinin2025gluon}
A.~Riabinin, E.~Shulgin, K.~Gruntkowska, and P.~Richt{\'a}rik.
\newblock {Gluon}: Making {Muon} \& {Scion} great again! (bridging theory and
  practice of {LMO}-based optimizers for {LLMs}).
\newblock {\em arXiv preprint arXiv:2505.13416}, 2025.

\bibitem{sato2025analysis}
N.~Sato, H.~Naganuma, and H.~Iiduka.
\newblock Analysis of {M}uon's convergence and critical batch size.
\newblock {\em arXiv preprint arXiv:2507.01598}, 2025.

\bibitem{semenov2025benchmarking}
A.~Semenov, M.~Pagliardini, and M.~Jaggi.
\newblock Benchmarking optimizers for large language model pretraining.
\newblock {\em arXiv preprint arXiv:2509.01440}, 2025.

\bibitem{sfyraki2025lions}
M.-E. Sfyraki and J.-K. Wang.
\newblock Lions and {M}uons: {O}ptimization via stochastic {F}rank-{W}olfe.
\newblock {\em arXiv preprint arXiv:2506.04192}, 2025.

\bibitem{essentialai2025practical}
I.~Shah, A.~M. Polloreno, K.~Stratos, P.~Monk, A.~Chaluvaraju, A.~Hojel, A.~Ma,
  A.~Thomas, A.~Tanwer, D.~J. Shah, et~al.
\newblock Practical efficiency of {Muon} for pretraining.
\newblock {\em arXiv preprint arXiv:2505.02222}, 2025.

\bibitem{shen2025convergence}
W.~Shen, R.~Huang, M.~Huang, C.~Shen, and J.~Zhang.
\newblock On the convergence analysis of {M}uon.
\newblock {\em arXiv preprint arXiv:2505.23737}, 2025.

\bibitem{sun2020optimization}
R.-Y. Sun.
\newblock Optimization for deep learning: {A}n overview.
\newblock {\em Journal of the Operations Research Society of China},
  8(2):249--294, 2020.

\bibitem{kimi2025k2}
K.~Team.
\newblock Kimi {K}2: {O}pen agentic intelligence.
\newblock {\em arXiv preprint arXiv:2507.20534}, 2025.
\newblock Tech report of Kimi K2.

\bibitem{tuddenham2022orthogonalising}
M.~Tuddenham, A.~Pr{\"u}gel-Bennett, and J.~Hare.
\newblock Orthogonalising gradients to speed up neural network optimisation.
\newblock {\em arXiv preprint arXiv:2202.07052}, 2022.

\bibitem{vyas2024soap}
N.~Vyas, D.~Morwani, R.~Zhao, I.~Shapira, D.~Brandfonbrener, L.~Janson, and
  S.~Kakade.
\newblock {SOAP}: {I}mproving and stabilizing {Shampoo} using {Adam}.
\newblock In {\em International Conference on Learning Representations (ICLR)},
  2025.

\bibitem{wang2025muon}
S.~Wang, F.~Zhang, J.~Li, C.~Du, C.~Du, T.~Pang, Z.~Yang, M.~Hong, and V.~Y.
  Tan.
\newblock Muon outperforms {A}dam in tail-end associative memory learning.
\newblock {\em arXiv preprint arXiv:2509.26030}, 2025.

\bibitem{wen2025fantastic}
K.~Wen, D.~Hall, T.~Ma, and P.~Liang.
\newblock Fantastic pretraining optimizers and where to find them.
\newblock {\em arXiv preprint arXiv:2509.02046}, 2025.

\bibitem{wu2018sgd}
L.~Wu, C.~Ma, and W.~E.
\newblock How {SGD} selects the global minima in over-parameterized learning:
  {A} dynamical stability perspective.
\newblock {\em Advances in Neural Information Processing Systems}, 31, 2018.

\bibitem{zhang2025adagrad}
M.~Zhang, Y.~Liu, and H.~Schaeffer.
\newblock Adagrad meets {M}uon: {A}daptive stepsizes for orthogonal updates.
\newblock {\em arXiv preprint arXiv:2509.02981}, 2025.

\bibitem{zhang2025provable}
X.~Zhang and H.~Gao.
\newblock On provable benefits of {M}uon in federated learning.
\newblock {\em arXiv preprint arXiv:2510.03866}, 2025.

\bibitem{zhang2024transformers}
Y.~Zhang, C.~Chen, T.~Ding, Z.~Li, R.~Sun, and Z.-Q. Luo.
\newblock Why transformers need {A}dam: {A} {H}essian perspective.
\newblock In {\em Advances in Neural Information Processing Systems (NeurIPS)},
  2024.

\end{thebibliography}

\end{document}